\documentclass[12pt]{article}

\makeatletter
\newfont{\footsc}{cmcsc10 at 8truept}
\newfont{\footbf}{cmbx10 at 8truept}
\newfont{\footrm}{cmr10 at 10truept}
\makeatother
\usepackage{amsmath}
\usepackage{amssymb}
\usepackage{exscale}
\usepackage{amsthm}
\usepackage{epsfig}
\usepackage{verbatim}
\usepackage{setspace}
\usepackage{bbm}
\usepackage{tikz}

\theoremstyle{plain}
\newtheorem{theorem}{Theorem}[section]
\newtheorem{proposition}[theorem]{Proposition}
\newtheorem{lemma}[theorem]{Lemma}

\theoremstyle{definition}
\newtheorem{definition}[theorem]{Definition}

\newtheorem{remark}[theorem]{Remark}
\newtheorem{example}[theorem]{Example}

\DeclareMathOperator{\1}{\bf{1}}
\DeclareMathOperator{\Var}{Var}

\DeclareMathOperator{\vhk}{Var_{HK}}
\DeclareMathOperator{\Prob}{{\bf P}}

\DeclareMathOperator{\modulo}{mod}

\newcommand{\N}{{\mathbb{N}}}

\newcommand{\R}{{\mathbb{R}}}

\newcommand{\E}{{\mathbb{E}}}

\def\bs0{\bf 0}

\title{On Negatively Dependent Sampling Schemes, Variance Reduction, and Probabilistic Upper Discrepancy Bounds}

\author{Marcin Wnuk\thanks{Mathematisches Seminar, Universit\"at Osnabr\"uck,
Germany ({\tt marcin.wnuk@math.uni-osnabrueck.de}).}
\and Michael Gnewuch\thanks{Institut f\"ur Mathematik, Universit\"at Osnabr\"uck,
Germany ({\tt michael.gnewuch@uni-osnabrueck.de}).}
\and Nils Hebbinghaus\thanks{Institut f\"ur Informatik, Christian-Albrechts-Universit\"at zu Kiel, Germany
({\tt Nils.Hebbinghaus@gmx.de}).}}
\begin{document}

\maketitle
\vskip 1pc


\begin{abstract}
We study some notions of negative dependence of a sampling scheme that can be used to derive variance bounds for the corresponding estimator or discrepancy bounds for the underlying random point set that are at least as good as the corresponding bounds for plain Monte Carlo sampling. 

We provide  new pre-asymptotic bounds with explicit constants for the star discrepancy and the weighted star discrepancy of sampling schemes that satisfy suitable negative dependence properties. Furthermore, we compare the different notions of negative dependence and give several examples of negatively dependent sampling schemes, including mixed sequences. 

\end{abstract}

\section{Introduction}

Plain Monte Carlo (MC) sampling is a method frequently used in stochastic simulation and multivariate numerical integration.
Let $p_1, \ldots, p_N$ be independent random points, each uniformly distributed in
the $d$-dimensional unit cube $[0,1)^d$. For an arbitrary integrable random variable (or function) $f: [0,1)^d \to \R$ we consider the MC estimator (or quadrature)
\begin{equation}\label{MC_estimator}
\mu^{\rm MC}(f) = \frac{1}{N} \sum_{i=1}^N f(p_i)
\end{equation}
for the expected value (or integral)
\begin{equation*}
I(f) = \int_{[0,1)^d} f(u) \, du.
\end{equation*}

An advantage of the MC estimator is that already under the very mild assumption on $f$ to be square integrable, it  converges to $I(f)$ for $N\to \infty$ with convergence rate $1/2$.
Even though the convergence rate is not very impressive, it has the invaluable advantage that it does not depend on the number of variables $d$.

However, there are many dependent sampling schemes (i.e., random sample points $p_i$, $i=1, \ldots, N$, that are still uniformly distributed in $[0,1)^d$, but not necessarily independent any more)  known that are superior to plain MC sampling with respect to  certain objectives.
An example are suitably randomized quasi-Monte Carlo (RQMC) point sets.
They ensure, for instance, higher convergence rates  for numerical integration of sufficiently smooth functions, they lead to much smaller  asymptotic discrepancy measures, their sample points do not tend to cluster and have more evenly distributed lower dimensional projections (see, e.g., \cite{DKS13, DP10, Lem09, Mat10}).
It would be desirable to be able to construct dependent sampling schemes that have some of these or other favorable properties, and that are, with respect to other objectives, at least as good as MC sampling schemes.

Recently, in this direction some research had been done.
In \cite{Lem18} Christiane Lemieux showed that  a negative dependence property of RQMC points ensures that the variance of the corresponding RQMC estimator for functions $f$ that are monotone with respect to each variable is never larger than the variance of the corresponding MC estimator $\mu^{\rm MC}f$. She also proved that a different negative dependence property yields that the variance of the RQMC estimator for an arbitrary bounded quasi-monotone $f$ is never larger than the variance of $\mu^{\rm MC}f$.
Those negative dependence properties rely solely on the marginals and the bivariate copulas of the RQMC points (i.e., on the distribution of single points and on the common distribution of pairs of points). Related results can be found in \cite{WG19}.

 In a different line of research the second and the third author of this book chapter showed in \cite{GH18, Heb12} that a specific negative dependence property of RQMC points guarantees that they satisfy the same pre-asymptotic probabilistic discrepancy bounds (with explicitly  revealed dependence on the number of points $N$ as well as on the dimension $d$) as MC points.
Here the negative dependence property relies on
the common distribution of all sample points.
Related results can be found in \cite{DDG18}.

For more extensive motivations of both lines of research we refer to the elaborate introductions of \cite{Lem18} and \cite{DDG18, GH18}, respectively.
The aim of this book chapter is to survey and compare the approaches mentioned above and to provide several new results.

This chapter is organized as follows: In Section~\ref{SECT2} we introduce some notions of negatively dependent sampling schemes and discuss how one can benefit from them.
In Section~\ref{SECT3} we provide new probabilistic upper discrepancy bounds for sampling schemes. The discrepancy measures we consider are the star discrepancy and the weighted star discrepancy. These bounds are ``plug-in results'' in the following sense: One just has to check whether a sampling scheme satisfies the sufficient negative dependence condition and -- if this is the case -- obtains immediately a probabilistic discrepancy bound with explicitly given constants.
In the Section~\ref{SECT4}, we give several examples of sampling schemes
that satisfy the one or the other notion of negative dependence, including a generalized notion of stratified sampling schemes and mixed randomized sequences. 
Finally, in the Section~\ref{SECT5} we elaborate on relations between different notions of negative dependence.

We finish the introduction by stating some notation.
Let $d,N \in \mathbb{N}.$ If not stated otherwise we are always considering a randomized point set $(p_j)_{j = 1}^N := \mathcal{P} \subset [0,1)^d$ consisting of $N$ points. For $a,b \in \mathbb{R}^d,a = (a_1,\ldots, a_d), b = (b_1,\ldots, b_d)$ we write $a \leq b$ if $a_i \leq b_i, i = 1, \ldots, d.$ All other inequalities are also to be understood componentwise. Moreover, $[a,b) := [a_1,b_1) \times \ldots \times [a_d,b_d)$. Via $ \mathcal{C}^d_0 $ we denote the set of boxes (``corners'') anchored at $0$
$$\mathcal{C}^d_0 := \{[0,a) \, | \,  a \in [0,1)^d \}, $$
and by $\mathcal{C}_1^d$ the set of boxes anchored at $1,$
$$\mathcal{C}^d_1 := \{[a,1) \,  | \,  a \in [0,1)^d \}. $$
We write $\mathcal{D}_0^d$ for the set of differences of boxes anchored at $0$,
$$\mathcal{D}_0^d := \{Q \setminus R \,  | \,  Q,R \in \mathcal{C}_0^d\}.$$
For $m\in \N$ we denote the set $\{1,2,\ldots,m\}$ by $[m]$, $\lambda^d$ stands for the d-dimensional Lebesgue measure on $\mathbb{R}^d,$ in case $d = 1$ we just write $\lambda.$ If not specified, all random variables are defined on a
probability space $(\Omega, \Sigma, \Prob)$.

\section{Review of Notions of Negative Dependance of Sampling Schemes}
\label{SECT2}

\subsection{$\gamma$-Negative Dependence of Binary Random Variables and Sampling Schemes}
The concept of negative dependence was introduced by Lehmann
\cite{Leh66} for pairs of random variables.
In the literature one finds several contributions on rather demanding
notions of negative dependence as, e.g., negative association
introduced in \cite{JDP83}; a survey can be found in \cite{Pem00}.
Sufficient for our purpose is the following notion
for \emph{Bernoulli} or \emph{binary random variables}, i.e.,
random variables that only take
values in $\{0,1\}$.

\begin{definition}\label{Def_Neg_gamma_Dep}
Let $\gamma \ge 1$.
We call binary random variables
$T_1,T_2,\ldots,T_N$
\emph{upper $\gamma$-negatively dependent} if
\begin{equation}
\label{cond2'}
\Prob\left( \bigcap_{j\in u} \{ T_j = 1 \}\right)
\le \gamma \prod_{j\in u} \Prob( T_j = 1)
\hspace{2ex}\text{for all $u \subseteq [N]$,}
\end{equation}
and \emph{lower $\gamma$-negatively dependent} if
\begin{equation}
\label{cond2''}
\Prob\left( \bigcap_{j\in u} \{T_j = 0 \}\right)
\le \gamma \prod_{j\in u} \Prob(T_j = 0)
\hspace{2ex}\text{for all $u \subseteq [N]$.}
\end{equation}
We call $T_1,T_2,\ldots,T_N$ \emph{$\gamma$-negatively dependent} if
both conditions (\ref{cond2'}) and (\ref{cond2''}) are satisfied.
If $\gamma = 1$, we usually suppress the explicit reference to
$\gamma$.
\end{definition}

$1$-Negative dependence is usually called negative orthant dependence, cf. \cite{BSS82}.

Notice that, in particular, independent binary random variables are negatively dependent.
Furthermore, it is easily seen that for $N=2$ and $\gamma =1$
the notions of upper and lower $\gamma$-negative dependence are
equivalent, cf. \cite{Leh66}.

We are interested in binary random variables $T_i$, $i=1, \ldots,N$, of the form
$T_i = \1_A(p_i)$, where $A$ is a Lebesgue-measurable subset of $[0,1)^d$ 
(whose characteristic function
is denoted by $\1_A$),  and $p_1, \ldots, p_N$ are randomly chosen points in $[0,1)^d$.

We will use the following bound of Hoeffding-type; for a proof see, e.g., \cite{Heb12}. 

\begin{theorem}\label{Hoeffding}
Let $\gamma \ge 1$, and let $T_1, \ldots , T_N$ be
$\gamma$-negatively dependent binary random variables. Put $S:= \sum^N_{i=1} (T_i-\E[T_i])$.
We have
\begin{eqnarray}\label{Hoeffding_both}
\Prob\left( |S| \geq t\right)  \leq 2\gamma\exp\left(-\frac{2t^2}{N}\right)
\hspace{2ex}\text{for all $t>0$.}
\end{eqnarray}
\end{theorem}

\begin{definition}
A randomized point set $\mathcal{P} = (p_j)_{j = 1}^N$ is called a \emph{sampling scheme} if every single $p \in \mathcal{P}$ is distributed uniformly in $[0,1)^d$ and the vector $(p_1, \ldots, p_N)$ is exchangeable, meaning that for any permutation $\pi$ of $[N]$ it holds that the law of $(p_1, \ldots, p_N)$ is the same as the law of $(p_{\pi(1)}, \ldots, p_{\pi(N)}).$
\end{definition}
The assumption of exchangeability is only of technical nature and, if we consider $\mathcal{P}$ as a randomized point set, may be always obtained in the process of symmetrization. Indeed, let $\tilde{\mathcal{P}}$ be a point set such that every $\tilde{p} \in \tilde{\mathcal{P}}$ is uniformly distributed in $[0,1)^d$ and let $\pi$ be a random uniformly chosen permutation of $[N].$ Then $(\tilde{p}_{\pi(1)}, \ldots, \tilde{p}_{\pi(N)})$ is already a sampling scheme.

\subsection{Pairwise Negative Dependence and Variance Reduction}
\begin{definition}
We say that a sampling scheme $\mathcal{P}$ is \emph{pairwise negatively dependent}  if for every $Q,R \in \mathcal{C}_1^d$ it holds that the random variables
$$\1_Q(p_1), \1_R(p_2) $$
are negatively dependent.
In other words, a sampling scheme $\mathcal{P}$ is pairwise negatively dependent if for every $Q,R \in \mathcal{C}_1^d$ we have 
\begin{equation}\label{pNLOD}
 \Prob(p_1 \in Q, p_2 \in R) \leq \Prob(p_1 \in Q) \Prob(p_2 \in R),
\end{equation}
\begin{equation}\label{pNUOD}
\Prob(p_1 \notin Q, p_2 \notin R) \leq \Prob(p_1 \notin Q) \Prob(p_2 \notin R).
\end{equation}
\end{definition}
Note  that (\ref{pNLOD}) implies (\ref{pNUOD}) and vice versa, therefore one of the conditions is in fact redundant.
In \cite{Lem18} this is known as negatively upper orthant dependent - or NUOD - sampling schemes.

Our interest lies in numerical integration of functions from some class $\mathcal{F} \subset L^2([0,1)^d)$ using RQMC. A QMC quadrature is just a quadrature consisting of $N$ nodes, such that the evaluation in every node is given the same weight $\tfrac{1}{N}.$ By randomizing the set of nodes we obtain an RQMC quadrature.
Let $\mu_{\mathcal{P}}f$ be an RQMC estimator of $I(f) := \int_{[0,1)^d} f(u) \, du$ based on the sampling scheme $\mathcal{P} = (p_i)_{i = 1}^N,$ i.e.
$$\mu_{\mathcal{P}}f = \frac{1}{N} \sum_{i = 1}^N  f(p_i).$$
Moreover, let $\mu^{\rm{MC}}f$ be an estimator of $I(f)$ based on a Monte Carlo sample consisting of $N$ points (i.e. the integration nodes are chosen independently and uniformly from $[0,1)^d$, see (\ref{MC_estimator})).

It turns out that randomized QMC quadratures based on  pairwise negatively dependent sampling schemes may lead to variance reduction in comparison to the simple MC quadratures.  Here we describe shortly one of such cases, namely when integrands are bounded quasimonotone functions. The following exposition is based on \cite{Lem18}.

To define what a quasimonotone function is we need first to introduce the notion of quasivolume.
For $a,b \in [0,1)^d,$ $J \subset [d]$ and a function $f:[0,1)^d \rightarrow \mathbb{R}$ we write $f(a_J,b_{-J})$ to represent the evaluation of $f$ at the point $(x_1, \ldots, x_d),$ where $x_j = a_j$ for $j \in J$ and $x_j = b_j$ otherwise.
The quasivolume of $f$ over an interval $A = [a,b) \subset [0,1)^d$ is given by
$$\Delta^{d}(f,A) := \sum_{J \subset [d]} (-1)^{|J|} f(a_J, b_{-J}).$$
We say that the function $f$ is \emph{quasimonotone} if
$$\Delta^d(f,A) \geq 0 $$
for every interval $A.$
Note that if we define a content $\nu_{f}([0,a)) := f(a), a \in [0,1)^d$ then quasimonotonicity of $f$ means exactly that for any axis-parallel rectangle $R \subset [0,1)^d$ it holds $\nu_f(R) \geq 0.$

Apart from pairwise negative dependence there are a few similar notions which are also of interest. Let $p_j = (p_j^{(1)},\ldots, p_j^{(d)}), j = 1, \ldots, N.$ If for every $i = 1, \ldots, d,$ and every measurable $A,B \subset [0,1)^{i-1}, \alpha,\beta \in [0,1)$ 
\begin{align*}
&\Prob(p_1^{(i)} \geq \alpha, p_2^{(i)} \geq \beta | p_1^{(1:i-1)} \in A, p_2^{(1:i-1)} \in B)
\\
&\leq \Prob(p_1^{(i)} \geq \alpha| p_1^{(1:i-1)} \in A, p_2^{(1:i-1)} \in B) \Prob(p_2^{(i)} \geq \beta| p_1^{(1:i-1)} \in A, p_2^{(1:i-1)} \in B), 
\end{align*}
we say that the sampling scheme $(p_j)_{j = 1}^N$ is \emph{conditionally negatively quadrant dependent} (conditionally NQD). Here $p^{(1:i-1)}$ denotes the orthogonal projection of $p$ onto its first $i-1$ coordinates. Note that the conditional NQD property holds in particular if $(p_1^{(i)}, p_2^{(i)})_{i = 1}^d$ are independent and for every $i = 1, \ldots, d,$  and every $q,r \in [0,1)$ we have
$$\Prob(p_1^{(i)} \in [q,1), p_2^{(i)} \in [r,1)) \leq \Prob(p_1^{(i)} \in [q,1)) \Prob(p_2^{(i)} \in [r,1)),$$
in which case we talk of a \emph{coordinatewise independent NQD sampling scheme}. Christiane Lemieux showed in \cite[Corollary 2]{Lem18} that conditionally NQD sampling schemes provide RQMC estimators of integrals with variance no bigger then the variance of the MC estimator if the integrand is monotone in each coordinate.

The following is basically a combination of Proposition $3,$ Remark $8$ and Corollary $2$ from \cite{Lem18}.
\begin{theorem}\label{VarRedThm}
Let $f:[0,1)^d \rightarrow \mathbb{R}$ and $\mathcal{P}$ be a  sampling scheme. Then if either
\begin{enumerate}
\item The function $f$ is bounded and $f$ or $-f$ is quasimonotone and $\mathcal{P}$ is pairwise negatively dependent,
\item The function $f$ is monotone in each coordinate and $\mathcal{P}$ is conditionally negatively quadrant dependent,
\end{enumerate}
it holds
$$\Var(\mu_{\mathcal{P}}f) \leq \Var(\mu^{\rm{MC}}f).$$
\end{theorem}

In Section~\ref{SECT5} we discuss relations between the introduced notions of negative dependence.

Let us note that the aforementioned paper provides actually more general results. Interested reader will find details in Sections 3 and 4 of \cite{Lem18}.

For examples of pairwise negatively dependent and conditionally NQD sampling schemes see Sections \ref{R1L} and \ref{LHS}.

\subsection{Negatively Dependent Sampling Schemes and Discrepancy}
\label{PDB}
\begin{definition}
We say that a sampling scheme $(p_j)_{j = 1}^N = \mathcal{P}$ is $\mathcal{S}-\gamma-$\emph{negatively dependent} if for every $Q \in \mathcal{S}$ the random variables
$$(\1_Q(p_j))_{j = 1}^N $$
are $\gamma$-negatively dependent.
In other words for every $t  \leq N$ we require
\begin{equation}\label{NUODDef}
\Prob(\bigcap_{j = 1}^t \{ p_j \in Q\}) \leq \gamma \prod_{j = 1}^t \Prob(p_j \in Q),
\end{equation}
\begin{equation}\label{NLODDef}
\Prob(\bigcap_{j = 1}^t \{p_j \notin Q)\} \leq \gamma \prod_{j = 1}^t \Prob(p_j \notin Q).
\end{equation}
\end{definition}
Note that differently from the case of pairwise negative dependence, for $N > 2$ one indeed needs to check both inequalities as they do not, in general, imply one another. If $\gamma = 1$ and $\mathcal{S} = \mathcal{C}_{0}^d$ we usually talk just of negatively dependent sampling schemes. Moreover, if (\ref{NUODDef}) is satisfied we speak of upper $\gamma-$ negatively dependent sampling schemes and if (\ref{NLODDef}) is satisfied we speak of lower $\gamma-$ negatively dependent sampling schemes.

To motivate the interest in negatively dependent sampling schemes we introduce the notion of discrepancy. Discrepancy is meant to quantify how far is a finite point set $P \subset [0,1)^d$ consisting of $N$ points from being equidistributed in $[0,1)^d.$ It plays an important role in fields like numerical integration, computer graphics, empirical process theory and many more. Let $x \in [0,1]^d $ and  $Q_x := [0,x) \in \mathcal{C}_0^d.$ We define the \emph{discrepancy function} $D_N(P,\cdot)$ for the point set $P$ at the point $x$ via
 $$D_N(P,x) := D_N(P, Q_x) := \left| \frac{1}{N} |P \cap Q_x| - \lambda^d(Q_x)\right| $$
and the \emph{star discrepancy} $D_N^*(P)$ by
$$D_N^*(P) := \sup_{x \in [0,1]^d} D_N(P,x).$$

 Making a connection to numerical integration we note one of the versions of the Koksma-Hlawka inequality, which states that for every point set $P$ consisting of $N$ points it holds
$$\left| \int_{[0,1)^d} f \,  d\lambda^d(x) - \frac{1}{N} \sum_{p \in P} f(p) \right| \leq D_N^*(P) \vhk(f),$$
where $\vhk(f)$ is the Hardy-Krause variation of $f$. The inequality is actually sharp, cf. \cite{Nie92}.

It has been shown in \cite{GH18} that $\mathcal{D}_0^d$ - $\gamma$ - negatively dependent sampling schemes have with large probability star discrepancy of the order $\sqrt{\tfrac{d}{N}}.$ More precisely the following theorem holds.

\begin{theorem}\label{main_theo_GH}
Let $d, N\in \N$ and $\rho \in [0,\infty)$. Let $\mathcal{P}= (p_j)_{j = 1}^N$ be a negatively $\mathcal{D}_0^d$-$e^{\rho d}$-dependent sampling scheme. \\
Then for every $c>0$
\begin{equation}\label{disc_bound_GH}
D^*_N(\mathcal{P}) \leq c\sqrt{\frac{d}{N}}
\end{equation}
holds with probability at least $1-e^{-(1.6741\cdot c^2 -10.7042- \rho)\cdot d}.$ Moreover, for every $\theta \in (0,1)$
\begin{equation}\label{disc_bound_AH+}
\Prob\left( D^*_N(\mathcal{P}) \leq 0.7729 \sqrt{ 10.7042 + \rho + \frac{\ln \big( (1-\theta)^{-1} \big) }{d}} \sqrt{\frac{d}{N}} \right) \geq \theta.
\end{equation}
\end{theorem}

Notice that these bounds depend only mildly on $\rho$ or $\gamma = e^{\rho d}.$ In particular, $\mathcal{D}_0^d$-$1$-negatively dependent sampling schemes satisfy the same preasymptotic discrepancy bound as Monte Carlo point sets do. For more details see \cite{GH18}.

In Remark~\ref{Remark_Counterpart} we present a bound similar to \eqref{disc_bound_AH+}
under a bit different assumptions that can be applied to so-called mixed randomized sequences.

\section{New Probabilistic Discrepancy Bounds}
\label{SECT3}

\subsection{Bound on the Star Discrepancy for Negatively Dependent Sampling Schemes}
Proving that a given sampling scheme is $\mathcal{D}_0^d$-$\gamma$-negatively dependent may turn out to be a difficult task. One of the problems lies in the fact that elements of $\mathcal{D}_0^d$ may in general not be represented as Cartesian products of one-dimensional  intervals, cf. also Remark~\ref{Remark_Counterpart}. With this in mind we would like to weaken the assumptions on the sampling scheme $\mathcal{P}.$ In the following result we show that by requiring  the sampling scheme $\mathcal{P}$ only to be $\mathcal{C}^d_0$-$\gamma$-negatively dependent one already gets with high probability a discrepancy of the order $\sqrt{\tfrac{d}{N} \log(e + \tfrac{N}{d}) }.$

 \begin{theorem}\label{main_theo}
Let $d, N\in \N$ and $\rho \in [0,\infty)$. Let $\mathcal{P}= (p_j)_{j=1}^N$ be a $\mathcal{C}^d_0$-$e^{\rho d}$-negatively dependent sampling scheme in $[0,1)^d.$ 
Then for every $c > 0$
\begin{equation}\label{NegDepBoxC}
D^*_N(\mathcal{P}) \leq c\sqrt{\tfrac{d}{N} \max \left\{ 1, \log\left( \frac{N}{d} \right)  \right\}\ }
\end{equation}
holds with probability at least $1 - 2 e^{(-\frac{1}{2}(c^2 - 1)\xi + \rho + \log(2e(\frac{2}{c} +1)))d},$
where $\xi = \max\left\{1, \log\left( \frac{N}{d} \right) \right\}.$
Moreover, for every $\theta \in (0,1)$
\begin{equation}\label{NegDepBoxTheta}
\Prob\left(D^{*}_N(\mathcal{P}) \leq \sqrt{\frac{2}{N}} \sqrt{d \log(\eta) + \rho d + \log\left( \frac{2}{1 - \theta} \right)} \right) \geq \theta,
\end{equation}
where $\eta:= \eta(N,d) = 6e \left( \max(1, \frac{N}{2d \log(6e)})  \right)^{\frac{1}{2}}.$
\end{theorem}
The proof of Theorem \ref{main_theo} requires some preparation.
To ``discretize'' the star discrepancy, we define $\delta$--covers 
as in~\cite{DGS05}: for any $\delta\in(0,1]$ a finite set $\Gamma$ of points in $[0,1)^d$ is called a \emph{$\delta$--cover} of $[0,1)^d$, if for every $y\in [0,1)^d$ there exist $x,z\in\Gamma\cup\{0\}$ such that $x\leq y\leq z$ and $\lambda^d([0,z])-\lambda^d([0,x])\leq\delta$. The number $\mathcal{N}(d,\delta)$ denotes the smallest cardinality of a
$\delta$--cover of $[0,1)^d$.

The following theorem was stated and proved in~\cite{Gne08a}.

\begin{theorem}
\label{bracketing}
For any $d\geq 1$ and $\delta\in(0,1]$ we have
\begin{equation*}
\mathcal{N}(d,\delta)\leq 2^d \frac{d^d}{d!}(\delta^{-1}+1)^d.
\end{equation*}
\end{theorem}

Notice that due to Stirling's formula we have $d^d/d! \le e^d/\sqrt{2\pi d}$ and so the cardinality of the $\delta-$cover may be bounded from above by $(2e)^d(1 + \delta^{-1})^d.$
Furthermore, it is easy to verify that in the case $d=1$ the identity
\begin{equation}\label{cover_d=1}
\mathcal{N}(1,\delta) = \lceil \delta^{-1} \rceil
\end{equation}
is established with the help of the $\delta$-cover $\Gamma := \{1/ \lceil \delta^{-1} \rceil, 2/ \lceil \delta^{-1} \rceil, \ldots,1\}$.

With the help of $\delta$-covers the star discrepancy can be approximated in the following sense.

\begin{lemma}\label{delta_approx}
Let $P\subset [0,1)^d$ be an $N$-point set, $\delta >0$, and $\Gamma$ be a $\delta$-cover  of $[0,1)^d$. Then
\begin{equation*}
D^*_N(P) \le \max_{x\in \Gamma} D_N(P, [0,x)) + \delta.
\end{equation*}
\end{lemma}

The proof of Lemma \ref{delta_approx} is straightforward, cf., e.g., \cite[Lemma~3.1]{DGS05}.

Now we are ready to prove Theorem \ref{main_theo}.
\begin{proof}
For $\delta \in (0,1)$ to be chosen later let $\Gamma$ be a $\delta-$cover consisting of at most $(2e)^d(1 + \delta^{-1})^d$ elements. Such a $\Gamma$ exists due to Theorem \ref{bracketing} and discussion thereafter.

Define
$$D^{*}_{N, \Gamma}(\mathcal{P}) = \max_{\beta \in \Gamma} \left|\lambda^d([0, \beta)) - \frac{1}{N} \sum_{j = 1}^N \1_{[0, \beta)}(p_j)\right|.$$
Now Lemma \ref{delta_approx} gives us
$$D^{*}_{N}(\mathcal{P}) \leq D^{*}_{N,\Gamma}(\mathcal{P}) + \delta. $$
For every $\beta \in \Gamma$ and $j \in [N]$ put
$$ \xi_{\beta}^{(j)} = \lambda^d([0, \beta)) - \1_{[0, \beta)}(p_j).$$
Let $\epsilon = 2\delta.$ Due to Hoeffding's inequality applied to random variables $(\xi_{\beta}^{(j)})_{j = 1}^N$ (applicable since $(p_j)_{j = 1}^N$ is $e^{\rho d}$ - negatively dependent) we obtain for every $\beta \in \Gamma$
$$\Prob\left(\left|\frac{\sum_{j = 1}^{N} \xi_{\beta}^{(j)} }{N}   \right| \geq \delta  \right) \leq 2e^{\rho d} e^{-2N\delta^2}.$$
With the help of a simple union bound we get
\begin{align}
& \Prob(D^{*}_{N}(\mathcal{P}) < \epsilon) = 1 - \Prob(D^{*}_{N}(\mathcal{P}) \geq \epsilon) \nonumber
\\
& \geq 1 - \Prob(D^{*}_{N,\Gamma}(\mathcal{P}) \geq \epsilon - \delta) = 1 - \Prob\left(\max_{\beta \in \Gamma} \left| \frac{\sum_{j = 1}^N\xi_{\beta}^{(j)} }{N}    \right| \geq \delta \right) \nonumber
\\
& \geq 1 - 2e^{\rho d} |\Gamma| e^{-\frac{N}{2} \epsilon^2}. \label{HoeffdingUnionBound}
\end{align}

Using the above we would like to find a bound on discrepancy of the sampling scheme $\mathcal{P}$ which holds with probability at least $\theta \in (0,1).$ We are looking for $\epsilon_{\theta}$ such that
\begin{equation}\label{thetaBound}
\Prob(D^{*}_{N}(\mathcal{P}) < \epsilon_{\theta}) \geq \theta.
\end{equation}
Put $\epsilon_{\theta} = C_{\theta} (\frac{d}{N} \log(1 + \frac{N}{d}))^{\frac{1}{2}} = 2\delta_{\theta}.$ Inequality (\ref{thetaBound}) holds true if $$\delta_{\theta} \geq \left( \frac{1}{2N} \right)^{\frac{1}{2}} \left[ \log(|\Gamma|) + \rho d + \log\left(\frac{2}{1-\theta}  \right)  \right]^{\frac{1}{2}}.$$
Our problem boils now down to finding possibly small $\delta_{\theta} \in (0,1)$ for which
\begin{equation}\label{thetaBound2}
\delta_{\theta} \geq  \left( \frac{1}{2N} \right)^{\frac{1}{2}} \left[ d\log\left(2e\left[1 + \delta_{\theta}^{-1}\right]\right) + \rho d + \log\left(\frac{2}{1-\theta}  \right)  \right]^{\frac{1}{2}}.
\end{equation}
Specifying $\delta_{\theta}$ to be of the form
$$\delta_{\theta} = \left( \frac{1}{2N} \right)^{\frac{1}{2}} \left[ d\log(\eta) + \rho d + \log\left(\frac{2}{1-\theta}  \right)  \right]^{\frac{1}{2}}$$
we get that (\ref{thetaBound2}) is satisfied if
$$\eta \geq 2e\left(1 + \delta_{\theta}^{-1}\right).$$
Expanding $\delta_{\theta}$ in dependence of $\eta$ it suffices to find $\eta$ for which
$$\left(\frac{\eta}{2e} - 1 \right)\log(\eta)^{\frac{1}{2}} \geq \left(\frac{2N}{d}  \right)^{\frac{1}{2}}$$
and one easily sees that this is satisfied for $\eta$ given in the statement of the theorem. To prove \ref{NegDepBoxC} one only needs to plug in $\epsilon = c \sqrt{\frac{d}{N}}$ into (\ref{HoeffdingUnionBound}) and then consider the two cases $\xi = 1$ and $\xi = \log\left( \frac{N}{d}  \right)$ separately. 
\end{proof}

\subsection{Bound on the Weighted Star Discrepancy for $\mathcal{D}_0^d-\gamma$-negatively Dependent Sampling Schemes.}
One of the reasons why the QMC integration may be successfully applied in many high-dimensional problems is the fact that quite often only a small number of coordinates is really important. This observation led to the introduction of weighted function spaces and weighted discrepancies by Sloan and Wo\'zniakowski in \cite{SW98}. The above concepts are closely related to the theory of weighted spaces of Sobolev type, in particular the integration error in those spaces obeys a Koksma-Hlawka type upper bound, which may be phrased using the norm of the function and the weighted star discrepancy.

By weights we understand a set of non-negative numbers $\gamma = (\gamma_u)_{u \in [d] \setminus \emptyset},$ where $\gamma_u$ is interpreted as the weight of the coordinates from $u.$ Let $|u|$ denote the cardinality of $u.$ For $x \in [0,1]^d$ we write $(x(u), 1)$ to denote the point in $[0,1]^d$ agreeing with $x$ on the coordinates from $u$ and having all the other coordinates set to $1.$

The weighted star discrepancy of a point set $X = (x_1, \ldots, x_N)$ and weights $\gamma$ is defined by
$$D_{N,\gamma}^*(X) := \sup_{z \in [0,1]^d} \max_{u \in [d] \setminus \emptyset} \gamma_u |D_N(X,(z(u),1)|.$$

The following theorem is similar in flavor to the Theorem $1$ from \cite{Ais11}.

\begin{theorem}
Let $N,d \in \mathbb{N}$ and let $\mathcal{P} = (p_j)_{j = 1}^N \subset [0,1)^d$ be a sampling scheme, such that for every $\emptyset \neq u \subset [d]$ its projection on the  coordinates in $u$ is $\mathcal{D}_0^{|u|}$ - $e^{\rho |u|}$- negatively dependent.
Then for any weights $(\gamma_u)_{u \subset [d] \setminus \emptyset}$ and any $c > 0$ it holds
\begin{equation}\label{NegDepWeightC}
D^*_{N, \gamma}(\mathcal{P}) \leq \max_{\emptyset \neq u \subset [d]} c \gamma_u \sqrt{\frac{|u|}{N}}
\end{equation}
with probability at least $2-(1+e^{-(1.674c^2 -10.7042 - \rho)})^d.$
Moreover, for $\theta \in (0,1)$ it holds
\begin{equation}\label{NegDepWeghtTheta}
\Prob\left(D^*_{N, \gamma}(\mathcal{P}) \leq \max_{\emptyset \neq u \subset [d]}  \gamma_u \sqrt{\tfrac{|\rho + 10.7 + \log((2-\theta)^{\tfrac{1}{d}} - 1 )|}{1.674}} \sqrt{\frac{|u|}{N}} \right) \geq \theta.
\end{equation}
\end{theorem}
\begin{proof}
We shall only prove the statement (\ref{NegDepWeightC}), the statement (\ref{NegDepWeghtTheta}) follows then by simple calculations.
For $\emptyset \neq u \subset [d]$ and $c > 0$ put
$$A_u = \{\omega \in \Omega : D^{*}_{N}(X^u(\omega))  > c \sqrt{\tfrac{|u|}{N}} \}.$$
Here $X^u$ denotes the projection of $X$ on the coordinates from $u.$
By Theorem \ref{main_theo_GH} it holds
\begin{align*}
\Prob(A_u) < e^{-(1.674c^2 - 10.7042 - \rho)|u|}.
\end{align*}
Now
\begin{align*}
& \Prob\left(D^*_{N, \gamma}(\mathcal{P}) > \max_{\emptyset \neq u \subset [d]} c \gamma_u \sqrt{\frac{|u|}{N}}\right) \leq \Prob(\bigcup_{\emptyset \neq u \subset [d]} A_u)
\\
& < \sum_{\nu = 1}^{d} \binom{d}{\nu}  e^{-(1.674c^2 -10.7042 - \rho) \nu}
\\
& = (1+e^{-(1.674c^2 -10.7042 - \rho)})^d - 1.
\end{align*}

\end{proof}

\section{Examples of Negatively Dependent and Pairwise Negatively Dependent Sampling Schemes}
\label{SECT4}

Many sampling schemes, such as randomly shifted and jittered rank-1 lattices (cf. Section \ref{R1L}) and Latin hypercube sampling (cf. Section \ref{LHS}), are multidimensional generalizations of the one-dimensional \emph{simple stratified sampling}. Simple startified sampling is defined in the following way: let $\pi$ be a uniformly chosen permutation of $\{1, \ldots, N\}$ and let $(U_j)_{j = 1}^N$ be independent random variables distributed uniformly on $(0,1].$ Moreover, $\pi$ is independent of $(U_j)_j.$ We put
$$p_j := \frac{\pi(j) - U_j}{N}, j = 1, \ldots, N. $$
Effectively, one is considering the partition $I_j := [\tfrac{j-1}{N}, \tfrac{j}{N}), j = 1, \ldots, N,$ of the unit interval and in every element of the partition putting one point, independently of all the other points. The simple lemma is a useful tool for our investigations and may be found e.g. in \cite{WG19}.

\begin{lemma}\label{SSSpND}
Simple stratified sampling $\mathcal{P} = (p_j)_{j = 1}^N$ is pairwise negatively dependent.
\end{lemma}

\subsection{Negative Dependence  of Generalized Stratified Sampling}\label{GSS}
We partition $[0,1)^d$ into $\beta \geq N$ sets $(B_j)_{j = 1}^{\beta}$ with $\lambda^d(B_j) = \frac{1}{\beta}, j = 1, \ldots, \beta.$ Let $Y=(Y_1, \ldots, Y_{\beta})$ be a random vector distributed uniformly on
$$\{ (v_1,\ldots, v_{\beta}) \in\{0,1\}^{\beta} : \sum_{j = 1}^{\beta} v_j = N\}.$$
Given the value of $Y$ we place one point for each $j \in [\beta]$ with $Y_j = 1$ uniformly and independently of all other points inside $B_j.$ Symmetrizing this construction yields a sampling scheme $\mathcal{P} = (p_j)_{j = 1}^N,$ which we call \emph{generalized stratified sampling} (note that every single $p \in \mathcal{P}$ is uniformly distributed in $[0,1)^d)$ . Here ``generalized''  has to be understood in the sense that there are possibly more strata then points.
\begin{example}
 There are many natural choices for the strata. The simplest one would be stripes of the form $B_j, j = 1, \ldots, N,$ with $B_j := [\tfrac{j-1}{N}, \tfrac{j}{N})\times [0,1)^{d-1}.$ However, one could also choose, e.g., elementary cells (i.e., fundamental parallelepipeds) of a rank-1 lattice (cf. \cite{LL00}), see Figure \ref{r1lStrata}. 
\begin{figure}
\begin{center}
\begin{tikzpicture}
 \draw(0,0)--(0,5);
 \draw(0,0)--(5,0);
 \draw(5,0)--(5,5);
 \draw(0,5)--(5,5);
 
  \draw(5/2,0)--(5,5);
  \draw(0,0)--(5/2,5);
  \draw(0,5/2)--(5,0);
  \draw(0,5)--(5,5/2);
\fill[gray] (4,3)--(5,5)--(5,5/2);
\fill[gray]  (0,5/2)--(1,2)--(2,4);
\fill[gray] (0,5/2)--(2,4)--(0,5);
\filldraw
(1/3,1.5) circle (2pt); 
\filldraw
(3,5/2) circle (2pt);
\filldraw
(7/2,1/2) circle (2pt); 
\filldraw
(37/20,97/20) circle (2pt); 
\filldraw
(4.7,2.8) circle (2pt); 
\end{tikzpicture}
\end{center}
\caption{$N$ elementary cells of a rank-1 lattice as strata, $\beta = N = 5$.}
\label{r1lStrata}
\end{figure}
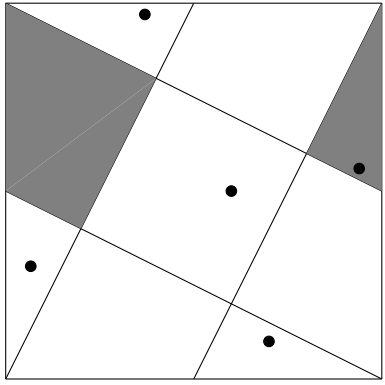
 
\end{example}
To show that generalized stratified sampling is negatively dependent we need first a simple lemma.
\begin{lemma}\label{MaxLemma}
Let $t, N \in \mathbb{N}, t \leq N, \xi \geq 0$ and let
$$D = \{x= (x_1, \ldots, x_N) \in \mathbb{R}^N \, | \, x \geq 0, \sum_{i = 1}^{N}x_i = \xi \}.$$
The function
$$f : D \rightarrow \mathbb{R}, (x_1, \ldots, x_N) \mapsto \sum_{J \subset [N], |J| = t} \prod_{j \in J} x_j $$
takes on its maximum in the point $(x_1, \ldots, x_N) = (\frac{\xi}{N}. \ldots, \frac{\xi}{N}).$
\end{lemma}
\begin{proof}
We shall prove the statement by induction on $N \geq t.$ The case $N = t$ is straightforward by Lagrange multipliers theorem. Suppose we have already shown the statement for $N-1$ and we would like to prove it for $N.$ Firstly let us fix the value of $x_N \in (0, \xi).$ It holds
\begin{align*}
& \sum_{J \subset [N], |J| = t} \prod_{j \in J} x_j = \sum_{J \subset [N], |J| = t, N \in J} \prod_{j \in J} x_j +  \sum_{J \subset [N], |J| = t, N \not\in J} \prod_{j \in J} x_j
\\
& = x_N \sum_{J' \subset [N-1], |J'| = t-1} \prod_{j \in J'} x_j + \sum_{J' \subset [N-1], |J'| = t} \prod_{j \in J'} x_j.
\end{align*}
By the induction assumption for a fixed value of $x_N$ the last term is maximal when for $j = 1,\ldots, N-1$ we have $x_j = \frac{\eta}{N-1},$ where we put $\eta = \xi - x_N.$  Plugging it into the above formula we obtain
$$ \sum_{J \subset [N], |J| = t} \prod_{j \in J} x_j = (\xi - \eta)\binom{N-1}{t-1}\left(\frac{\eta}{N-1}\right)^{t-1} + \binom{N-1}{t} \left(\frac{\eta}{N-1}\right)^t,$$
which we need to maximize with respect to $\eta.$
It holds
$$ \sum_{J \subset [N], |J| = t} \prod_{j \in J} x_j = Ch(\eta),$$
where $C = \frac{(N-1)!}{(t-1)!(N-t)!(N-1)^{t-1}}$ and $h(\eta) = \xi \eta^{t-1} + \left(\frac{N-t}{t(N-1)} -1 \right)\eta^t.$
Now we have
$$h'(\eta) = \eta^{t-2}\left[(t-1)\xi + \left(\frac{N-t}{N-1} - t \right)\eta \right].$$
The derivative vanishes for $t \geq 3$ at $\eta_1 = 0$ and $\eta_2 = \frac{N-1}{N} \xi.$ Since $h(\eta_2) > \max\{ h(0), h(\xi) \}$ and $\eta_2$ is a local maximum the claim follows.
\end{proof}

\begin{theorem}\label{GSSThm}
Let $\mathcal{P} = (p_j)_{j = 1}^N$ be a generalized stratified sampling as described above and $A \subset [0,1)^d$ be measurable. Then for every $1 \leq t \leq N$ it holds
$$\Prob(\bigcap_{j = 1}^t \{ p_j \in A\}) \leq \prod_{j = 1}^t \Prob(p_j \in A).$$
 In particular, generalized stratified sampling is $\mathcal{S}$ - negatively dependent for any system $\mathcal{S}$ of measurable subsets of $[0,1)^d$.
\end{theorem}
\begin{proof}
Fix $t$ as in the statement of the theorem and define
$$D_t = \{(k_1,\ldots, k_t) \in [\beta]^t : \forall_{i,j \in [t]} \hspace{3mm} i \neq j \implies k_i \neq k_j \}.$$
Note that $|D_t| = \beta(\beta-1) \cdots (\beta-t+1).$
For $k = (k_1,\ldots, k_t) \in D_t$ we have
$$\Prob(\bigcap_{j = 1}^t \{Y_{k_j} = 1\}) = \frac{\binom{\beta - t}{N-t}}{\binom{\beta}{N}} = \frac{N(N-1)\cdots (N-t+1)}{\beta(\beta-1)\cdots(\beta - t + 1)}.$$
By Lemma \ref{MaxLemma} it follows
\begin{align*}
& \Prob(\bigcap_{j = 1}^t \{p_j \in A\}) = \sum_{k \in D_t} \Prob(\bigcap_{j = 1}^t p_j \in A| \bigcap_{j = 1}^t \{p_j \in B_{k_j}\}) \Prob(\bigcap_{j = 1}^t \{p_j \in B_{k_j} \} | \bigcap_{j = 1}^t \{Y_{k_j} = 1\})\Prob(\bigcap_{j = 1}^t \{Y_{k_j} = 1\})
\\
& = \sum_{k \in D_t} \prod_{j = 1}^t \frac{\lambda^d(A \cap B_{k_j})}{\lambda^d(B_{k_j})} \frac{1}{N(N-1)\cdots (N-t+1)} \frac{N(N-1)\cdots (N-t+1)}{\beta(\beta-1)\cdots(\beta - t + 1)}
\\
& = \frac{1}{\beta(\beta -1 ) \cdots (\beta - t + 1)} \sum_{k \in D_t} \prod_{j = 1}^t \frac{\lambda^d(A \cap B_{k_j})}{\lambda^d(B_{k_j})}
\\
& \leq \frac{1}{\beta(\beta -1 ) \cdots (\beta - t + 1)} \beta(\beta -1 ) \cdots (\beta - t + 1) \left( \frac{\lambda^d(A)}{\beta} \right)^t \beta^t = (\lambda^d(A))^t.
\end{align*}
\end{proof}

\begin{remark}\label{NoPNDSSS}
 Without further information on the strata we cannot make any conclusions about pairwise negative dependence of generalized stratified sampling. As an example consider a stratified sampling scheme $\mathcal{P} = (p_1,p_2)$  defined by two strata $B_1,B_2$ in $d \geq 2.$ One may choose $B_1,B_2$ and $Q,R \in \mathcal{C}_1^d$ in such a way that $Q \subset B_1, B_2 \subset R$ and $R \neq [0,1)^d,$ see Figure \ref{NoPNDSSSFig}. In this case however
 $$\Prob(p_2 \in R| p_1 \in Q) = 1, $$
 and the sampling scheme is not pairwise negatively dependent.
 
 On the other hand if we consider strata $B_j, j = 1, \ldots, N,$ with $B_j := [\tfrac{j-1}{N}, \tfrac{j}{N}) \times [0,1)^{d-1}$ then this practically boils down to the one-dimensional case and so the corresponding sampling scheme is pairwise negatively dependent, cf. Lemma \ref{SSSpND}. 
\end{remark}

 \begin{figure}
  \begin{center}
  \begin{tikzpicture}
    \draw(0,0)--(0,5);
    \draw(0,0)--(5,0);
    \draw(5,0)--(5,5);
    \draw(0,5)--(5,5);
    
    \draw(2,0)--(2,4)--(5,4);
     \filldraw 
(3.5,2) circle (0.1pt) node[below] {$B_2$};
    \filldraw 
(1,2.5) circle (0.1pt) node[below] {$B_1$};
    \end{tikzpicture}
  \end{center}
\caption{Example of strata for Remark \ref{NoPNDSSS}}
\label{NoPNDSSSFig}
\end{figure}
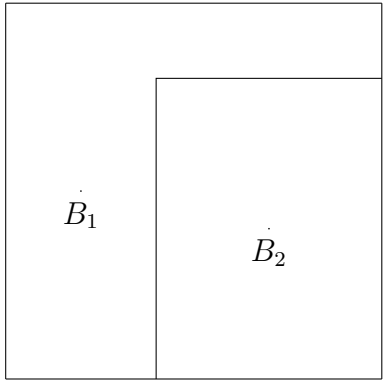

\subsection{Pairwise Negative Dependence and Conditional NQD Property of Randomly Shifted and Jittered Rank-1 Lattices}\label{R1L}
The exposition follows closely \cite{WG19}.
Let $N$ be prime. By $\mathbb{F} := \mathbb{F}_{N}$ we denote $\{0,1,\ldots, N-1\}$. Moreover, $\mathbb{F}^* := \mathbb{F} \setminus \{0\}.$ We also put $\widetilde{\mathbb{F}} := \frac{1}{N} \mathbb{F}$ and similarly $\widetilde{\mathbb{F}}^* := \frac{1}{N}\mathbb{F}^*.$

A discrete subgroup $\mathcal{L}$ of the $d-$dimensional torus $\mathbb{T}^d$ is called a lattice. A set $(y_j)_{j = 1}^N$ is a rank-1 lattice if for some $g \in (\tilde{\mathbb{F}}^*)^d$ it admits a representation
$$y_j = (j-1)g \hspace{2mm}  \modulo 1 \quad  j = 1,\ldots, N.$$ In this case $g$ is called a generating vector of the lattice.

Note that our definition differs from the usual one in that we allow only for generating vectors $g$ from $(\widetilde{\mathbb{F}}^*)^d$ and not from $\widetilde{\mathbb{F}}^d,$ which saves us from considering some degenerate cases.

We want now to define a sampling scheme based on rank-1 lattices which we call randomly shifted and jittered rank-1 lattice. To this end let $(y_j)_{j = 1}^N$ be a rank-1 lattice with generating vector chosen randomly uniformly from $(\widetilde{\mathbb{F}}^*)^d.$ Let $U$ be distributed uniformly on $\widetilde{\mathbb{F}}^d,$ $J_j, j = 1, \ldots, N$ be uniformly distributed on $[0, \frac{1}{N})^d$ and $\pi$ be a uniformly chosen permutation of $\{1,\ldots, N\}.$ Moreover, let all of the aforementioned random variables be independent. We put
$$p_{j}:= y_{\pi(j)} + U + J_j \text{ mod } 1, j = 1, \ldots, N.$$
We call the sampling scheme $\mathcal{P} = (p_j)_{j = 1}^N$ a \emph{randomly shifted and jittered rank-1 lattice} (\emph{RSJ rank-1 lattice}).
Putting it in words: we first take a rank-1 lattice with a random generator and symmetrize it. Then we shift the lattice uniformly on the torus, where the shift has resolution $\frac{1}{N}$. In the last step we jitter every point independently of all the other points in a cube of volume $(\tfrac{1}{N})^d$.

The following is Theorem 3.4. from \cite{WG19}.
\begin{theorem}\label{PairNegDep}
Let $N$ be prime, $d \in \mathbb{N}$. RSJ rank-1 lattice $\mathcal{P} = (p_j)_{j = 1}^N$ in $[0,1)^d$ is a coordinatewise independent NQD sampling scheme.
\end{theorem}
In particular, RSJ rank-1 lattice  is a pairwise negatively dependent and a conditionally NQD sampling scheme, which means that both alternative conditions for $\mathcal{P}$ from Theorem \ref{VarRedThm} hold.

In contrast to generalized stratified sampling (cf. Theorem \ref{GSSThm}) and Latin hypercube sampling (see Theorem \ref{LHSThm}), RSJ rank-1 lattice is for $d \geq 2$ and $N \geq 3$ in general not $\mathcal{C}_0^d$- negatively dependendent, see Subsection \ref{SECT5.1.}.

\subsection{Negative Dependence, Conditional NQD Property, and Pairwise Negative Dependence of Latin Hypercube Sampling}\label{LHS}
Let $(\pi_i)_{i = 1}^d$ be independent uniformly chosen permutations of $[N],$ and $U^{(i)}_j, i = 1,\ldots, d, j = 1, \ldots, N$ be independent random variables distributed uniformly on $(0,1]$ and independent also of the permutations. A sampling scheme $(p_j)_{j=1}^N$ is called a \emph{Latin hypercube sampling} if the $i-$th coordinate of the $j-$th point $p_j^{(i)}$ is given by
$$p_j^{(i)} = \frac{\pi_{i}(j) - U^{(i)}_j}{N}, i = 1,\ldots, d, j = 1,\ldots, N.$$
What one intuitively does is the following: one cuts $[0,1)^d$ into slices $(S_{k,j})_{j = 1}^N, k = 1, \ldots, d$ given by
$$S_{k,j} = \prod_{j = 1}^{k-1}[0,1) \times [\tfrac{j-1}{N},\tfrac{j}{N}) \times \prod_{j = k+1}^d [0,1) $$
and puts $N$ points in such a way that in every slice there is exactly one point.

It is worth mentioning that for $d = 1$ Latin hypercube sampling is exactly the same as RSJ rank-1 lattice (namely simple stratified sampling). For $d \geq 2$ the joint distribution of a pair of points is the same for Latin hypercube sampling as for RSJ rank-1 lattice. But if we sample more than two points, then the joint distributions already differ, see \cite{WG19}.

Negative dependence of Latin hypercube Sample has been studied in \cite{GH18} and pairwise negative dependence has been investigated in \cite{WG19}.
\begin{theorem}\label{LHSThm}
Latin hypercube Sample in $[0,1)^d$ is a sampling scheme which is
\begin{itemize}
\item[(i)] $\mathcal{D}^d_0$ - $e^{d}$ - negatively dependent,
\item[(ii)] $\mathcal{C}_0^d$ - negatively dependent,
\item[(iii)] coordinatewise independent NQD.
\end{itemize}
\end{theorem}
In the above, statements $(i)$ and $(ii)$ follow from Theorem $3.5.$ from \cite{GH18}, and statement $(iii)$ is Theorem $3.4.$ from \cite{WG19}.

In particular from $(iii)$ it follows that LHS is pairwise negatively dependent as well as conditionally NQD. 

\subsection{Pairwise Negative Dependence of Scrambled (0,m,s)-Nets.}

The so called $(t,m,s)$-nets belong to the most regular deterministic point sets. First defined by Niederreiter in \cite{Nie87}, they have been subject of extensive research. For a nice introduction on $(t,m,s)$-nets and their randomization, see \cite{Mat10}. 

Let us fix a base $b\in \mathbb{N}_{\geq 2}.$ For $j \in \mathbb{N}_0$ and $k = 0,1,\ldots, b^j-1$ an interval of the form
$$E_k^j = [kb^{-j}, (k+1)b^{-j})$$
is called an \emph{elementary interval (in base b)}. Moreover, for $s \in \mathbb{N}$ and vectors ${\bf{j}} = (j_1, \ldots, j_s)$ and ${\bf{k}} = (k_1,\ldots, k_s)$ (where for every $l = 1,\ldots, s,$ we require $0\leq k_l \leq b^{j_l} - 1$) we define an $s$-dimensional elementary interval via
$$E_{\bf{k}}^{{\bf{j}}} := \prod_{l = 1}^s E_{k_l}^{j_l}.$$

A \emph{$(t,m,s)$-net} is any $P \subset [0,1]^s$ such that for any elementary interval $E$ with $\lambda^s(E) = b^{-m + t}$ there are exactly $b^t$ point in $P \cap E.$ It is easily seen that a $(t,m,s)$-net consists of exactly $b^{m}$ points. Specific constructions of $(t,m,s)$-nets are known.

Scrambling of depth $m$ is a bijective function $S: [0,1]^s \rightarrow [0,1]^s$ such that for any elementary interval $E$ with $\lambda^s(E) = b^{-m}$ the image $S(E)$ is again an elementary interval of volume $b^{-m}.$ 

Now let us focus on the case $t = 0.$ Taking a $(0,m,s)$-net and applying to it a random scrambling of depth $m$ one obtains a randomized point set. Scramblings are defined in such a way that for any scrambling $S$ of depth $m$ and a $(0,m,s)$-net $P$, the point set $S(P)$ is again a $(0,m,s)$-net. By an appropriate choice of randomized scrambling $\tilde{S}$ one may make $\tilde{S}(P)$ to be a sampling scheme. In this case we call $\tilde{S}(P)$ a scrambled $(0,m,s)$-net. Scrambling as a way of randomization of $(0,m,s)$-nets has been studied by A.B. Owen, e.g., in \cite{Owe97b}.

In a recent article \cite{LW19} C. Lemieux and J. Wiart have shown the following theorem (which follows from Corollary $4.10$ from the aforementioned article).

\begin{theorem}
 Scrambled $(0,m,s)$-nets are pairwise negatively dependent sampling schemes.
\end{theorem}

\subsection{Mixed Randomized Sequences}

As already mentioned, part of the success of RQMC stems from the fact that in many high-dimensional practical integration problems only a small number of coordinates is of real importance. It stands to reason that one tries to use it to his avail by constructing quadratures in which one uses RQMC on the ``important'' coordinates and simple (usually much cheaper) Monte Carlo for the rest of the coordinates. This method is sometimes referred to as padding and the resulting sequences of integration nodes are called mixed sequences. Let us give a formal definition.

\begin{definition}
Let $d,d',d''\in\N$ with $d=d'+d''$. Let $X=(X_k)_{k\in\N}$ be a
sequence in $[0,1)^{d'}$, and
let $Y=(Y_k)_{k\in\N}$ be a sequence 
in $[0,1)^{d''}$. The $d$-dimensional concatenated
sequence $Z=(Z_k)_{k\in \N} = (X_k, Y_k)_{k\in \N}$
is called a \emph{mixed sequence}. 
If $Y$ is a sequence of independent uniformly distributed random points, one also says that $Z$ results from $X$ by
\emph{padding by Monte Carlo} and calls $Z$ a \emph{hybrid-Monte Carlo sequence}.
If $X$ and $Y$ are both randomized sequences, we call $Z$ a \emph{mixed randomized sequence}.
\end{definition}

Padding by Monte Carlo was introduced by Spanier in \cite{Spa95} to tackle
problems in particle transport theory. He suggested to use a hybrid-Monte Carlo sequence resulting from padding a deterministic low-discrepancy sequence.
Hybrid-Monte Carlo sequences showed a favorable performance in several numerical experiments, see, e.g.,
\cite{Okt96, OTB06}. The latter papers also provided
theoretical results on probabilistic
discrepancy estimates of hybrid-Monte Carlo sequences which have been improved in
\cite{AH12, Gne09}. 
Favorable discrepancy bounds for padding Latin hypercube sampling (LHS) by Monte Carlo were provided in \cite{GH18}.
Padding a sequence by LHS (instead of by Monte Carlo) was considered earlier by Owen \cite[Example~5]{Owe94}.

A related line of research, initiated in \cite{Nie09}, is to study the discrepancy of concatenated sequences that result from two deterministic sequences.
More recent results can, e.g., be found in \cite{GPHN13, DHL17, Hof18} and the literature mentioned therein.

The following proposition shows that concatenating two mutually independent negatively dependent sampling schemes results again in a (higher dimensional) negatively dependent sampling scheme.
A weaker version of the next proposition may be found in \cite{Heb12}; cf. Lemma~5 there.

\begin{proposition}\label{Concatenation}
Let $d,d',d''\in \N$ such that $d=d'+d''$. Let $A \subseteq [0,1)^{d'}$,
$B \subseteq [0,1)^{d''}$ be Borel measurable sets. Let $x_1,\ldots,x_N$ be a sampling scheme in $[0,1)^{d'}$
and $y_1,\ldots,y_N$ a sampling scheme in $[0,1)^{d''}$. Furthermore, let $\alpha, \beta \geq 1$.
\begin{itemize}
\item[(i)] If the random variables
$\1_A(x_i)$, $i=1,\ldots,N$, and $\1_B(y_i)$, $i=1,\ldots,N$, are upper negatively $\alpha$- and 
$\beta$- dependent, respectively,
and mutually independent, then the random variables $\1_{A\times B}(x_i,y_i)$, $i=1,\ldots,N$,
induced by the random vectors
$(x_1,y_1),\ldots,(x_N,y_N)$ in $[0,1)^d$, are upper negatively $\alpha \beta$-dependent.
\item[(ii)] If the random variables
$\1_A(x_i)$, $i=1,\ldots,N$, and $\1_B(y_i)$, $i=1,\ldots,N$, are lower negatively $\alpha$- and $\beta$-dependent, respectively,
and mutually independent, then the random variables $\1_{A\times B}(x_i,y_i)$, $i=1,\ldots,N$,
induced by the random vectors
$(x_1,y_1)$, ... ,$(x_N,y_N)$ in $[0,1)^d$, are lower negatively $\alpha \beta$-dependent.
\end{itemize}
\end{proposition}


\begin{proof}
Let us first prove statement (i). Obviously we have for $J\subseteq [N]$
\begin{equation*}
\begin{split}
&\Prob \left( \bigcap_{j\in J} \{\1_{A\times B}(x_j,y_j) = 1 \}\right) = \Prob \left( \bigcap_{j\in J} \{x_j \in A\} \cap   \bigcap_{j\in J} \{y_j \in B \} \right)\\
&= \Prob \left( \bigcap_{j\in J} \{x_j \in A\} \right) \Prob \left( \bigcap_{j\in J} \{y_j \in B \}\right)
\le \left( \alpha \prod_{j\in J} \Prob(x_j \in A)\right) \left( \beta \prod_{j\in J}\Prob(y_j \in B) \right)\\
&= \alpha \beta \prod_{j\in J} \Prob \left( \1_{A\times B}(x_j,y_j) =1 \right).
\end{split}
\end{equation*}
We now prove statement (ii). 
Take any $\emptyset \neq J \subseteq [N]$ and set $t = |J|$ . Suppose first that $((x_j,y_j))_{j = 1}^N$ is a hybrid-Monte Carlo sequence, i.e. $(y_j)_{j = 1}^N$ is a Monte Carlo sampling scheme.
Due to our assumptions in statement~(ii) we obtain
\begin{equation*}
\begin{split}
&\Prob \left( \bigcap_{j\in J} \{\1_{A\times B}(x_j,y_j) = 0\} \right) 
\\
& = \sum_{K \subseteq J} \Prob\left(\bigcap_{j \in J} \{\1_{A \times B}(x_j,y_j) = 0\}   \cap \bigcap_{j \in K} \{\1_B(y_j) = 1 \}  \cap \bigcap_{j \in J \setminus K} \{\1_B(y_j) = 0 \}  \right) 
\\
& = \sum_{\nu = 0}^t \binom{t}{\nu} \Prob\left( \bigcap_{j = 1}^{\nu} \{\1_A(x_j) = 0 \}   \right) \Prob\left( \bigcap_{j = 1}^{\nu} \{\1_B(y_j) = 1 \}    \cap \bigcap_{j = \nu + 1}^{t} \{\1_B(y_j) = 0 \}\right) 
\\
&\leq \alpha \sum_{\nu = 0}^t \binom{t}{\nu} \Prob \left( \1_A(x_1) = 0 \right)^{\nu} \Prob\left( 1_B(y_1) = 1 \right)^{\nu} \Prob\left(  \1_B(y_1) = 0    \right)^{t - \nu}
\\
& = \alpha [\Prob(\1_A(x_1) = 0 ) \Prob(\1_B(y_1) = 1) + \Prob(\1_B(y_1) = 0 ) ]^t
= \alpha \Prob(\1_{A \times B}(x_1,y_1) = 0 )^t.
\end{split}
\end{equation*}
 Now let $(y_j)_{j = 1}^N$ be any sampling scheme in $[0,1)^{d''}$ such that  the random variables $(\1_B(y_j))_{j = 1}^N$ are  lower $\beta$-negatively dependent and let $(\hat{y}_j)_{j = 1}^N$ be a Monte Carlo sampling scheme in $[0,1)^{d''}$; we assume both sampling schemes to be mutually independent to $(x_j)_{j=1}^N$. Analogously as in the previous case we obtain
\begin{equation*}
 \begin{split}
&\Prob \left( \bigcap_{j\in J} \{\1_{A\times B}(x_j,y_j) = 0 \}\right)
\\
&= \sum_{\nu = 0}^t \binom{t}{\nu} \Prob\left(\bigcap_{j = 1}^{\nu} \{\1_A(x_j) = 1 \} \cap \bigcap_{j = \nu+1}^t \{\1_A(x_j) = 0 \} \right) \Prob\left(\bigcap_{j = 1}^{\nu} \{\1_B(y_j) = 0 \} \right)
\\
& \leq \sum_{\nu = 0}^t \binom{t}{\nu} \Prob\left(\bigcap_{j = 1}^{\nu} \{\1_A(x_j) = 1 \} \cap \bigcap_{j = \nu+1}^t \{\1_A(x_j) = 0 \} \right) \beta \Prob\left( \1_B(y_1) = 0 \right)^{\nu}
\\
& = \beta \sum_{\nu = 0}^t \binom{t}{\nu} \Prob\left(\bigcap_{j = 1}^{\nu} \{\1_A(x_j) = 1 \} \cap \bigcap_{j = \nu+1}^t \{\1_A(x_j) = 0 \} \right) \Prob\left( \1_B(\hat{y}_1) = 0 \right)^{\nu}.
 \end{split}
\end{equation*}
It follows from the case of hybrid-Monte Carlo sequences that
\begin{equation*}
 \begin{split}
  & \Prob\left( \bigcap_{j \in J} \{\1_{A \times B}(x_j,y_j) = 0 \} \right) \leq \beta \Prob\left( \bigcap_{j \in J} \{\1_{A \times B}(x_j, \hat{y}_j)= 0 \}\right)
  \\
  & \leq \alpha \beta \Prob\left(\1_{A \times B}(x_1, \hat{y}_1) = 0  \right)^t = \alpha \beta \Prob\left( \1_{A \times B}(x_1, y_1) = 0   \right)^t
 \end{split}
\end{equation*}

\end{proof}

\begin{remark}
It follows easily on closer examination of the proof that for the statement (i) of Proposition~\ref{Concatenation} to hold true we need only $(\1_A(x_j))_{j = 1}^N$ and $(\1_B(y_j))_{j = 1}^N$ to be negatively $\alpha$- respectively $\beta$-upper dependent point sets, not necessarily sampling schemes. Moreover, if in (ii) we assume that $(y_j)_{j = 1}^N$ is a Monte Carlo sampling scheme we also do not need to assume that $(x_j)_{j = 1}^N$ is a sampling scheme.
\end{remark}

\begin{remark}\label{Remark_Counterpart}
Let $\mathcal{S}'$, $\mathcal{S}''$ be systems of measurable sets in $[0,1)^{d'}$ and $[0,1)^{d''}$,
respectively. Let $(x_j)_{j=1}^N$ be an  $\mathcal{S}'$-$\alpha$-negative dependent sampling scheme in  $[0,1)^{d'}$ and  $(y_j)_{j=1}^N$ an   $\mathcal{S}''$-$\beta$-negative dependent sampling scheme in  $[0,1)^{d''}$; both sampling schemes should be mutually independent. Furthermore, let $\mathcal{P}:= (p_j)_{j=1}^N$ be the resulting concatenated sampling scheme in $[0,1)^d$, i.e., 
$p_i := (x_i,y_i)$, $i=1,\ldots,N$.
\begin{itemize}
\item[(i)] If $\mathcal{S}'=\mathcal{C}_0^{d'}$ and $\mathcal{S}'' =\mathcal{C}_0^{d''}$, we obtain from 
Proposition~\ref{Concatenation} that the mixed randomized sequence $(p_j)_{j=1}^N$ is $\mathcal{C}^d_0$-$\alpha \beta$-negatively dependent, which implies that we may directly apply Theorem~\ref{main_theo} to obtain a probabilistic discrepancy bound for $\mathcal{P}$.
\item[(ii)] If $\mathcal{S}'=\mathcal{D}_0^{d'}$ and $\mathcal{S}'' =\mathcal{D}_0^{d''}$, we
obtain from Proposition~\ref{Concatenation} that $(p_j)_{j=1}^N$ is $\alpha \beta$-negatively dependent with respect to the set system 
$$\mathcal{D}^{d'}_0 \times \mathcal{D}^{d''}_0 := \{ D'\times D''\,|\, D'\in \mathcal{D}^{d'}_0, 
D''\in \mathcal{D}^{d''}_0 \} \neq \mathcal{D}^{d}_0.$$
Hence Theorem~\ref{main_theo_GH} is unfortunately not directly applicable to $\mathcal{P}$.
Nevertheless, one may prove a counterpart of Theorem~\ref{main_theo_GH} with slightly worse constants that relies on negative dependence with respect to $\mathcal{D}^{d'}_0 \times \mathcal{D}^{d''}_0$. Namely, one may show for every $\theta \in (0,1)$ that 
\begin{equation}\label{disc_bound_mixed}
\Prob\left( D^*_N(\mathcal{P}) \leq 2 \ast 0.7729 \sqrt{ 10.7042 + \rho + \frac{\ln \big( (1-\theta)^{-1} \big) }{d}} \sqrt{\frac{d}{N}} \right) \geq \theta.
\end{equation}
The bound is based on the following simple observation: To estimate the local discrepancy of 
$\mathcal{P}$ in a test box $Q \in \mathcal{C}^d_0$, the strategy used in \cite{GH18} (and earlier in \cite{Ais11}) is to decompose $Q$ into finitely many disjoint differences of boxes $\Delta_1, \ldots, \Delta_K \in \mathcal{D}^d_0$ such that $Q= \cup_{\nu =1}^K \Delta_\nu$. This gives
\begin{equation}\label{erste_Zerlegung}
D_N(\mathcal{P}, Q) \le \sum_{\nu=1}^K D_N(\mathcal{P}, \Delta_\nu).
\end{equation}
Now let us consider a fixed index $\nu$. Then we find $A_\nu, B_\nu \in \mathcal{C}^d_0$ such that
$A_\nu \subseteq B_\nu$ and $\Delta_\nu = B_\nu \setminus A_\nu$. Furthermore, we may write $A_\nu = A_\nu'\times A_\nu''$ and $B_\nu = B_\nu' \times B_\nu''$ with $A_\nu', B_\nu' \in \mathcal{C}_0^{d'}$ and  $A_\nu'', B_\nu'' \in \mathcal{C}_0^{d''}$. Then we may represent $\Delta_\nu$ as  disjoint union
$$\Delta_\nu = (B_\nu' \setminus A_\nu') \times B_\nu'' \cup A_\nu' \times (B_\nu''\setminus A_\nu'') 
=: C^1_\nu \cup C^2_\nu.$$
Thus 
\begin{equation}\label{zweite_Zerlegung}
D_N(\mathcal{P}, \Delta_\nu) \le D_N(\mathcal{P}, C^1_\nu) + D_N(\mathcal{P}, C^2_\nu),
\end{equation}
where $C^1_\nu, C^2_\nu \in \mathcal{D}^{d'}_0 \times \mathcal{D}^{d''}_0$.
Now large deviation inequalities of Bernstein- and Hoeffding-type can be used to obtain for each of the random variables $D_N(\mathcal{P}, C^1_\nu)$, $D_N(\mathcal{P}, C^2_\nu)$ the same upper bound as for the local discrepancy $D_N(\mathcal{P}^*, \Delta_\nu)$ of a $\mathcal{D}^d_0$-$\alpha \beta$-negative dependent sampling schemes $\mathcal{P}^*$ in the proof of \cite[Theorem~4.3]{GH18}. This, combined with \eqref{erste_Zerlegung} und \eqref{zweite_Zerlegung},
results in a probabilistic discrepancy bound for $D^*_N(\mathcal{P})$ that is as most as twice as big as the one from Theorem~\ref{main_theo_GH}; for further details see \cite[Proof of Theorem~4.3]{GH18}.
\end{itemize}
\end{remark}

\section{Relations Between Notions of Negative Dependence}\label{SECT5}

It may be easily seen that the coordinatewise independent NQD property implies the pairwise negative dependence property as well as the conditional NQD property. It turns out that this is the only valid implication between the considered notions of negative dependence. In this section we give examples showing that other implications do not hold.
\subsection{Pairwise Negative Dependence and Negative Dependence}\label{SECT5.1.}
Neither the pairwise negative dependence of a sampling scheme implies the negative dependence, nor the other way round.

\begin{example}
We first show an example of a negatively dependent sampling scheme which is not pairwise negatively dependent. To this end consider a sampling scheme consisting of just two points $(p_1,p_2)$ with joint CDF $F:[0,1]^2 \rightarrow [0,1]$ given by
$$F(x,y) = \min\{x,y,\tfrac{x^2+y^2}{2}\}.$$
It is easy to see that $F(0,0) = 0, F(1,1) = 1, F$ is continuous, qusi-monotone, and $F(x,y) = F(y,x)$, which implies that $F$ is a CDF of a sampling scheme. Moreover,
$$\Prob\left(p_1 \in \left[0,q\right), p_2 \in \left[0,q\right)\right) = F(q,q) = q^2,$$
so the sampling scheme is $\mathcal{C}_0^1$ - negatively dependent. Notice that due to $d = 1,$ it is equivalent to saying that the sampling scheme is $\mathcal{C}_1^1$ - negatively dependent. However, for instance
\begin{align*}
&\Prob(p_1 \in [\tfrac{3}{4}, 1), p_2 \in [\tfrac{1}{4},1)) = 1-(F(\tfrac{3}{4},1) + F(1, \tfrac{1}{4}) - F(\tfrac{3}{4}, \tfrac{1}{4}))
\\
& = F(\tfrac{3}{4}, \tfrac{1}{4}) = \tfrac{1}{4} > (1 - \tfrac{3}{4})(1-\tfrac{1}{4}) = \Prob(p_1 \in [\tfrac{3}{4},1)) \Prob(p_2 \in [\tfrac{1}{4},1)).
\end{align*}
\end{example}

\begin{example}
To see that even the stronger coordinatewise independent NQD property does not imply the negative dependence property consider RSJ rank-1 lattice defined in Subsection \ref{R1L}. On the one hand, according to Theorem \ref{PairNegDep}, RSJ rank-1 lattice is coordinatewise independent NQD. On the other hand, let us consider the situation for $d = 2,$ and a large $N$ to be chosen later. We put $Q = [0, \tfrac{3}{N})^2.$
Obviously
$$\Prob(p_1 \in Q)^3 = \left( \frac{3}{N}  \right)^6.$$
We also have
$$\Prob(p_1 \in Q, p_2 \in Q, p_3 \in Q) \geq \frac{1}{\binom{N}{3}N(N-1)} = \frac{6}{N^2(N-1)^2(N-2)},$$
the inequality follows since for the diagonal configuration of the points (i.e. $p_j = (\tfrac{\pi(j)}{N},  \tfrac{\pi(j)}{N}) + J_j, j=1,\ldots, n$ for some permutation $\pi$ of $\{1, \ldots, N\}, k \in [N-1]$) there is one triple of points always lying in $Q.$ Notice that any generating vector of the form $g = (\tfrac{k}{N}, \tfrac{k}{N})$ and any shift of the form $S = (\tfrac{l}{N}, \tfrac{l}{N}), l \in \{0,1,\ldots, N-1\},$ results in a diagonal configuration.  
Now for $N$ large enough it holds
$$\Prob(p_1 \in Q, p_2 \in Q, p_3 \in Q) >  \Prob(p_1 \in Q)^3.$$
\end{example}

\subsection{Conditional NQD and Pairwise Negative Dependence}
\begin{example}
First we show an example of a pairwise negatively dependent sampling scheme which is not conditionally NQD.
Let $B_1 = [0,\tfrac{1}{2})^2, B_2 = [\tfrac{1}{2},1) \times [0, \tfrac{1}{2}), B_3 = [0, \tfrac{1}{2}) \times [\tfrac{1}{2},1), B_4 = [\tfrac{1}{2}, 1)^2$ denote the slots. Now we are considering a sampling scheme $\mathcal{P} = (p_1,p_2)$ such that given the slots the points are distributed uniformly within the slots and are independent. Denote $A_{ij} := \{p_1 \in B_i, p_2 \in B_j\}$ and set
$$\Prob(A_{ii}) = \tfrac{1}{16}, i = 1,2,3,4, $$
$$\Prob(A_{13}) = \Prob(A_{24}) = \Prob(A_{31}) = \Prob(A_{42}) = \tfrac{1}{32},$$
$$\Prob(A_{14}) = \Prob(A_{23}) = \Prob(A_{41}) = \Prob(A_{32}) = \tfrac{5}{32}.$$
It is easy to see that $\mathcal{P}$ is not conditionally NQD, e.g.
\begin{align*}
&\Prob(p_1^{(2)} \geq \tfrac{1}{2}, p_2^{(2)} \geq \tfrac{1}{2} | p_1^{(1)} \geq \tfrac{1}{2}, p_2^{(1)} \geq \tfrac{1}{2})
 = \tfrac{1}{3}
\\
 &> \tfrac{1}{4} = \Prob(p_1^{(2)} \geq \tfrac{1}{2}| p_1^{(1)} \geq \tfrac{1}{2}, p_2^{(1)} \geq \tfrac{1}{2}) \Prob(p_2^{(2)} \geq \tfrac{1}{2} | p_1^{(1)} \geq \tfrac{1}{2}, p_2^{(1)} \geq \tfrac{1}{2}).
\end{align*}
Showing that $\mathcal{P}$ is pairwise negatively dependent requires simple but tedious calculations and as such will be omitted. Intuitively it is clear, since the sampling scheme gives high probability to diagonal arrangements (i.e. $A_{14}, A_{23}, A_{41}, A_{32}$).
\end{example}
\begin{example}
Now we show an example of a sampling scheme which is conditionally NQD but not pairwise negatively dependentd. To this end let $X,Y$ be two independent random variables distributed uniformly on $[0,1).$ We consider a sampling scheme $\mathcal{P} = (p_1,p_2)$ given by $p_1 = (X,Y), p_2 = (Y,X).$ Let $u,v \in [0,1)^2$ and $A,B \subset [0,1)$ be measurable.
Sampling scheme $\mathcal{P}$ is conditionally NQD since
\begin{align*}
& \Prob(p_1^{(2)} \geq u^{(2)}, p_2^{(2)} \geq v^{(2)} |  p_1^{(1)} \in A, p_2^{(1)} \in B)
 \\
 &= \Prob(Y \geq u^{(2)}, X \geq v^{(2)} | X \in A, Y \in B)
\\
& = \Prob(Y \geq u^{(2)} | X \in A, Y \in B) \Prob(X \geq v^{(2)} | X \in A, Y \in B )
 \\
 &= \Prob(p_1^{(2)} \geq u^{(2)}| X \in A, Y \in B) \Prob(  p_2^{(2)} \geq v^{(2)} | X \in A, Y \in B).
\end{align*}
On the other hand $\mathcal{P}$ is not pairwise negatively dependent. To see this note that
\begin{align*}
&\Prob(p_1 \geq u)\Prob(p_2 \geq v)
\\
&= \Prob(X \geq u^{(1)}, Y \geq u^{(2)}) \Prob(Y \geq v^{(1)}, X \geq v^{(2)})
\\
&= \Prob(X \geq u^{(1)}) \Prob(Y \geq u^{(2)}) \Prob(Y \geq v^{(1)}) \Prob(X \geq v^{(2)})
\end{align*}
and
\begin{align*}
\Prob(p_1 \geq u, p_2 \geq v) = \Prob(X \geq \max(u^{(1)}, v^{(2)}), Y \geq \max(u^{(2)}, v^{(1)})).
\end{align*}
Taking for some $u^{(1)},u^{(2)} \in (0,1)$ the point $v$ satisfying $v^{(1)} = u^{(2)}$ and $v^{(2)} = u^{(1)}$ yields the claim. 
\end{example}

\bibliographystyle{siam}

\begin{thebibliography}{10}

\bibitem{Ais14}
{\sc C.~Aistleitner}, {\em Covering numbers, dyadic chaining and discrepancy},
  J.~Complexity, 27 (2011), pp.~531--540.

\bibitem{Ais11}
\leavevmode\vrule height 2pt depth -1.6pt width 23pt, {\em Tractability results
  for the weighted star-discrepancy}, J.~Complexity, 30 (2014), pp.~381--391.

\bibitem{AH12}
{\sc C.~Aistleitner and M.~T. Hofer}, {\em Probabilistic error bounds for the
  discrepancy of mixed sequences}, Monte Carlo Methods Appl., 18 (2012),
  pp.~181--200.

\bibitem{BSS82}
{\sc H.~W. Block, T.~H. Savits, and M.~Shaked.}, {\em Some concepts of negative
  dependence}, Ann. Probab., 10 (1982), pp.~765--772.

\bibitem{DKS13}
{\sc J.~Dick, F.~Y. Kuo, and I.~H. Sloan}, {\em High dimensional integration --
  the quasi-{M}onte {C}arlo way}, Acta Numerica, 22 (2013), pp.~133--288.

\bibitem{DP10}
{\sc J.~Dick and F.~Pillichshammer}, {\em Digital nets and sequences},
  Cambridge University Press, Cambridge, 2010.

\bibitem{DDG18}
{\sc B.~Doerr, C.~Doerr, and M.~Gnewuch}, {\em Probabilistic lower discrepancy
  bounds for {L}atin hypercube samples}, in Contemporary Computational
  Mathematics -- a Celebration of the 80th Birthday of Ian Sloan, J.~Dick,
  F.~Y. Kuo, and H.~Wo\'zniakowski, eds., Springer-Verlag, 2018, pp.~339--350.

\bibitem{DGS05}
{\sc B.~Doerr, M.~Gnewuch, and A.~Srivastav}, {\em Bounds and constructions for
  the star discrepancy via $\delta$-covers}, J. Complexity, 21 (2005),
  pp.~691--709.

\bibitem{DHL17}
{\sc M.~Drmota, R.~Hofer, and G.~Larcher}, {\em On the discrepancy of
  halton-kronecker sequences}, in Number Theory - Diophantine problems, Uniform
  Distribution and Applications - Festschrift in Honour of Robert F. Tichy's
  60th Birthday, 2017, pp.~219--226.

\bibitem{Gne08a}
{\sc M.~Gnewuch}, {\em Bracketing numbers for axis-parallel boxes and
  applications to geometric discrepancy}, J. Complexity, 24 (2008),
  pp.~154--172.

\bibitem{Gne09}
\leavevmode\vrule height 2pt depth -1.6pt width 23pt, {\em On probabilistic
  results for the discrepancy of a hybrid-{M}onte {C}arlo sequence}, J.
  Complexity, 25 (2008), pp.~312--317.

\bibitem{GH18}
{\sc M.~Gnewuch and N.Hebbinghaus}, {\em Discrepancy bounds for a class of
  negatively dependent random points including {L}atin hypercube samples}.
\newblock Preprint 2018.

\bibitem{GPHN13}
{\sc D.~Gomez-Perez, R.~Hofer, and H.Niederreiter}, {\em A general discrepancy
  bound for hybrid sequences involving {H}alton sequences}, Uniform
  Distribution Theory, 8 (2013), pp.~31--45.

\bibitem{Heb12}
{\sc N.~Hebbinghaus}, {\em Mixed sequences and application to multilevel
  algorithms}, Master's thesis, Christ Church, University of Oxford, 2012.

\bibitem{Hof18}
{\sc R.~Hofer}, {\em {K}ronecker-{H}alton sequences in $\mathbb{F}_{p}
  ((x^{-1}))$}, Finite Fields and Their Applications, 50 (2018), pp.~154--177.

\bibitem{JDP83}
{\sc K.~Joag-Dev and F.~Proschan.}, {\em Negative association of random
  variables, with applications}, Ann. Statist., 11 (1983), pp.~286--295.

\bibitem{LL00}
{\sc P.~L'Ecuyer and C.~Lemieux}, {\em Variance reduction via lattice rules},
  Management {S}cience, 46 (2000), pp.~1214--1235.

\bibitem{Leh66}
{\sc E.~Lehmann}, {\em Some concepts of dependence}, Ann. Math. Statist., 37
  (1966), pp.~1137--1153.

\bibitem{Lem09}
{\sc C.~Lemieux}, {\em Monte Carlo and Quasi-Monte Carlo Sampling}, Springer,
  New York, 2009.

\bibitem{Lem18}
\leavevmode\vrule height 2pt depth -1.6pt width 23pt, {\em Negative dependence,
  scrambled nets, and variance bounds}, Mathematics of {O}perations {R}esearch,
  43 (2018), pp.~228--251.

\bibitem{LW19}
{\sc J.~Wiart and C.~Lemieux}{\em On the dependence structure of scrambled (t,m,s)-nets},Preprint, ArXiV 1903.09877, 2019.
  
\bibitem{Mat10}
{\sc J.~Matou\v{s}ek}, {\em Geometric Discrepancy}, Springer-Verlag Berlin
  Heidelberg, 2010.

\bibitem{Nie87}
{\sc H.~Niederreiter}, {\em Point sets and sequences with small diescrepancy}, Monatsh. Math., 104 (1987), pp. 273-337.
  
\bibitem{Nie92}
{\sc H.~Niederreiter}, {\em Random Number Generation and Quasi-Monte Carlo
  Methods}, vol.~63 of CBMS-NSF Regional Conference Series in Applied
  Mathematics, Society for Industrial and Applied Mathematics (SIAM),
  Philadelphia, 1992.

\bibitem{Nie09}
{\sc H.~Niederreiter}, {\em On the discrepancy of some hybrid sequences}, Acta
  Arith., 138 (2009), pp.~373--398.

\bibitem{Okt96}
{\sc G.~\"Okten}, {\em A probabilistic result on the discrepancy of a
  hybrid-{M}onte {C}arlo sequence and applications}, Monte Carlo Methods Appl.,
  2 (1996), pp.~250--270.

\bibitem{OTB06}
{\sc G.~{\"O}kten, B.~Tuffin, and V.~Burago}, {\em A central limit theorem and
  improved error bounds for a hybrid-{M}onte {C}arlo sequence with applications
  in computational finance}, J. Complexity, 22 (2006), pp.~435--458.

\bibitem{Owe94}
{\sc A.~B. Owen}, {\em Lattice sampling revisited: {M}onte {C}arlo variance of
  means over randomized orthogonal arrays}, Ann. Statist., 22 (1994),
  pp.~930--945.

\bibitem{Owe97b}
{\sc A.~B. Owen}, {\em Monte {C}arlo variance of scrambled net quadrature}, SIAM J. Num. Anal., 34 (1997).
  
\bibitem{Pem00}
{\sc R.~Pemantle}, {\em Towards a theory of negative dependence}, Journal of
  Math. Physics, 41 (2000), pp.~1371--1390.

\bibitem{SW98}
{\sc I.~H. Sloan and H.~Wo\'{z}niakowski}, {\em When are quasi-{M}onte {C}arlo
  algorithms efficient for high dimensional integrals?}, J. Complexity, 14
  (1998), pp.~1--33.

\bibitem{Spa95}
{\sc J.~Spanier}, {\em Quasi-{M}onte {C}arlo methods for particle transport
  problems}, in Monte Carlo and Quasi-Monte Carlo Methods in Scientific
  Computing, H.~Niederreiter and P.~J.-S. Shiue, eds., Berlin, 1995,
  Springer-Verlag, pp.~121--148.

\bibitem{WG19}
{\sc M.~Wnuk and M.~Gnewuch}, {\em Note on pairwise negative dependence of
  randomized rank-$1$ lattices}.
\newblock Preprint, ArXiV: 1903.02261, 2019.

\end{thebibliography}

\end{document}